\documentclass[a4paper]{article}
\usepackage{amsthm,amssymb,amsmath,enumerate,graphicx,psfrag,color,pict2e,wrapfig}
\usepackage{algorithm}
\usepackage{tikz}

\usetikzlibrary{backgrounds}
\tikzstyle{edge} = [draw, line width = 1pt ,-]
\tikzstyle{thickedge} = [draw=blue, line width = 2pt ,-]
\tikzstyle{vertex} = [ shape=circle, fill=black, inner sep = 2pt]

\newtheorem{definition}{Definition}
\newtheorem*{claim}{Claim}

\newtheorem{theorem}[definition]{Theorem}

\newtheorem{lemma}[definition]{Lemma}

\newcommand{\comment}[1]{}

\newcommand{\es}{\emptyset}

\newcommand{\cA}{\mathcal{A}}
\newcommand{\cB}{\mathcal{B}}

\newcommand{\cG}{\mathcal{G}}
\newcommand{\cS}{\mathcal{S}}

\newcommand{\bP}{\mathbb{P}}
\newcommand{\bE}{\mathbb{E}}

\DeclareMathOperator{\bin}{Bin}

\newcommand{\sm}{\setminus}

\definecolor{red}{RGB}{255,0,0}

\title{Long paths and cycles in random subgraphs of graphs with large minimum degree}
\author{Stefan Ehard and Felix Joos}
\date{}

\begin{document}
\maketitle

\begin{abstract}
\noindent
For a graph $G$ and $p\in [0,1]$, 
let $G_p$ arise from $G$ by deleting every edge mutually independently with probability $1-p$.
The random graph model $(K_n)_p$ is certainly the most investigated random graph model and also known as the $G(n,p)$-model.
We show that several results concerning the length of the longest path/cycle naturally translate
to $G_p$ if $G$ is an arbitrary graph of minimum degree at least $n-1$.

For a constant $c$, we show that asymptotically almost surely 
the length of the longest path is at least $(1-(1+\epsilon(c))ce^{-c})n$
for some function $\epsilon(c)\to 0$ as $c\to \infty$,
and the length of the longest cycle is a least $(1-O(c^{- \frac{1}{5}}))n$.
The first result is asymptotically best-possible.
This extents several known results on the length of the longest path/cycle of a random graph in the $G(n,p)$-model.
\end{abstract}

\section{Introduction}

Around 1960 Erd\H{o}s and Renyi proved the first results about random graphs -- 
especially about graphs on $n$ vertices where every possible edge is present independently with probability $p$, 
which is nowadays known as the $G(n,p)$-model.
It is not an overstatement saying that this field has grown enormously since then
and for numerous graph parameters the typical value is (precisely) known for large $n$.
In particular, the lengths of paths and cycles are investigated.
As for any $\epsilon>0$ and $p \geq \frac{(1+ \epsilon) \log n}{n}$ a.a.s.\ a graph in $G(n,p)$ is hamiltonian,
we consider the length of a longest path/cycle if $p= \frac{c}{n}$ for some constant $c>1$.
A~series of papers~\cite{AKS81,Bol82,BFF84,Fri86,Veg79} finally led to the following theorem,
where 
$$\alpha(c)= \sup_{\alpha\geq 0}\{G\in G(n,cn^{-1}) \text{ contains a path of length }\alpha n \text{ a.a.s.}\}$$
and $\beta(c)$ analogously for the length of cycles.

\begin{theorem}
There exists a function $\epsilon(c) \rightarrow 0$ as $c \rightarrow \infty$ such that
\begin{align*}
	\alpha(c),\beta(c)= 1-(1+\epsilon(c))ce^{-c}.
\end{align*} 
\end{theorem}
Let us consider a more general random graph model.
For a graph $G$, we denote by $G_p$ the random subgraph obtained by deleting every edge independently with probability $1-p$ from the edge set of $G$. 
Thus $(K_{n})_p$ is a uniformly at random chosen member of $G(n,p)$.
In this paper we consider the typical asymptotic behavior of $(G_k)_p$ instead of $(K_n)_p$ where $G_k$ is a simple graph of minimum at least $k$.
In our setting $p$ depends on $k$ instead of the order of $G_k$.
We denote by $\cG$ the set of all graph sequences $G_1,G_2,\ldots$ such that $G_k$ has minimum degree at least $k$.
We define
\begin{align*}
	\alpha'(c)= \inf_{(G_k)_{k\geq 1}\in \cG} \ \sup_{\alpha\geq 0}\{(G_k)_p\text{ contains a path of length }\alpha k \text{ a.a.s.}\}
\end{align*}
and $\beta'(c)$ analogously for cycles.
It is clear that $\alpha'(c)\leq \alpha(c)$ and $\beta'(c)\leq \beta(c)$.
We prove that there is essentially no difference between $\alpha'(c)$ and $\alpha(c)$
and our second contribution is a lower bound on $\beta'(c)$.

\begin{theorem}\label{thm: path}
There exists a function $\epsilon(c) \rightarrow 0$ as $c \rightarrow \infty$ such that
\begin{align*}
	\alpha'(c)= 1-(1+\epsilon(c))ce^{-c}.
\end{align*}
\end{theorem}

\begin{theorem}\label{thm: cycle}
We have
\begin{align*}
	\beta'(c)=1-O\left(c^{- \frac{1}{5}}\right).
\end{align*}
\end{theorem}

Thus Theorem~\ref{thm: path} describes precisely the asymptotic behavior of $\alpha'(c)$ as $c \rightarrow \infty$
improving a result due to Krivelevich, Lee and Sudakov~\cite{KLS15} who showed that $\alpha'(c)=1-O(c^{- \frac{1}{2}})$.
In addition, it generlizes results concerning the length of the longest path in the $G(n,p)$-model due to 
Ajtai, Koml{\'o}s and Szemer{\'e}di~\cite{AKS81},
Fernandez de la Vega~\cite{Veg79},
Bollob\'as~\cite{Bol82},
Bollob\'as, Fenner and Frieze~\cite{BFF84}, and Frieze~\cite{Fri86}.   

Theorem~\ref{thm: cycle} improves a result of Krivelevich, Lee and Sudakov~\cite{KLS15} and Riordan~\cite{Rio14}
implying $\beta'(c)=1-o(1)$.
It also generalizes several results of the length of the longest cycle in the $G(n,p)$-model.

Note that the questions of hamiltonicity in the $G(n,p)$ setting
translates to the question whether $G_k$ has a cycle of length at least $k+1$.
These extensions are successfully settled by Krivelevich, Lee and Sudakov~\cite{KLS15}, and 
by Glebov and Naves and Sudakov~\cite{GNS14}.

\section{Preliminaries}

We will frequently need to show that a binomial random variable is very close to its expected value and 
use for these purposes Chernoff's inequality.
\begin{theorem}[Chernoff's inequality \cite{AS04}]\label{thm: cher}
If X is a binomial distributed random variable with \mbox{$X\sim \bin\left(n,p\right)$} and $0<\lambda \leq np= \bE X$, 
then 
\begin{align*}
	\bP\left[|X-np|\geq\lambda\right]\leq 2e^{-\frac{\lambda^2}{3np}}.
\end{align*}
\end{theorem}

Several results in this paper are based on the depth-first-search algorithm (DFS-algorithm) which is a frequently used exploration method of graphs. 
We briefly describe this algorithm and introduce some notation along the way.
Several recent results apply this algorithm to random graphs leading to very nice and short proofs~\cite{KLS15, KS13, Rio14}.

The DFS-algorithm is an algorithm traversing a graph such that all vertices of a given graph $G$ are finally visited and outputs a rooted spanning forest $T$ of $G$. 
It proceeds in the following way.

At any step, there is a partition of the vertex set $V(G)$ into three sets $R$, $S$ and $U$. 
The set $U$  contains the vertices that have not yet been visited during the exploration, 
$R$ denotes the set of vertices whose exploration is complete, 
and all the remaining vertices that are currently under exploration are contained in~$S$.  
The vertices of $S$ are kept in a stack, 
which is a last-in-first-out data structure.

The algorithm starts with $U=V(G)$ and $R=S=\emptyset$ 
and executes the following rounds until every vertex is explored, 
i.e. $R=V(G)$ and $S=U=\emptyset$.
\begin{itemize}
	\item If $S=\emptyset$, then some unreached vertex $v$ in $U$ is moved to $S$. 
	This vertex $v$ will be the root of a new component of our rooted spanning forest $T$.
	\item Otherwise, let $v$ be the top element of the stack $S$ (the \emph{last-in} vertex). 
	The algorithm queries whether $v$ has some neighbor $w$ in $U$. 
	If so, $w$ is placed on top of the stack $S$. 
	If $v$ has no neighbor in $U$, it is completely explored and is moved to $R$.
	\item As long as $U\neq \es$, the algorithm moves to the next round.
\end{itemize}

In each round of the algorithm there is exactly one vertex moved either from $U$ to $S$ or from $S$ to $R$. 
So indeed, after $2|V(G)|$ rounds every vertex has been moved from $U$ to $R$ through $S$ and the algorithm terminates with a rooted spanning forest $T$.

The following properties of the DFS-algorithm are important to us:
\begin{enumerate}[(I)]
	\item\label{DFS1}Every positively answered query about a neighbor in $U$ increases the size of $R\cup S$ by exactly one.
	\item\label{DFS2}The set $S$ always spans a path.
	\item\label{DFS3}At any round of the algorithm, all possible edges between the set $R$ and $U$ have been
 queried and answered negatively.
	\item\label{DFS4}Every edge $e=uv$ of the graph $G$ which is not tested during the exploration of $G$ joins two vertices on some vertical path in the rooted spanning forest $T$ (because otherwise the algorithm would have queried for the edge $uv$ during the exploration). 

\end{enumerate}

We will use the DFS-algorithm to explore the random graph $G_p$. 
Therefore, we assume that the algorithm already knows the underlying graph $G$ and all the edges of $G$. 
The DFS-algorithm only queries about these edges of $G$ during the exploration of $G_p$. 
That is, if the DFS-algorithm looks for neighbors of some vertex $v$, 
it only considers the neighbors $w$ of $v$ in $G$,
and queries whether this vertex is also a neighbor of $v$ in $G_p$. 
We receive a positive answer of each such query independently with probability $p$. 
In this way, following this algorithm,
we explore a rooted spanning forest of our random graph $G_p$.
Note that by definition the answer of a query does not depend on the answers of the previous queries.
We say an edge of $G$ is \emph{tested} if the DFS-algorithm queried whether this edge is in $G_p$ and otherwise we say it is \emph{untested}.

Throughout the paper we consider  graphs $G_k$ of minimum degree at least $k$.
Almost all our results include asymptotic statements
and an event occurs asymptotically almost surely (a.a.s.) if the probability that this event occurs tends to $1$ as \mbox{$k\to\infty$}.
Furthermore, several inequalities in our computations are only correct if $k$ is large enough
and for the purpose of readability we often drop the index $k$ and simple write $G$.

\section{Auxiliary Results}

Before we begin with the proofs of Theorem~\ref{thm: path} and \ref{thm: cycle}, 
we cite and prove some results for later use.
The first one uses a nice and direct analysis of the DFS-algorithm.

\begin{lemma}[Krivelevich, Lee, Sudakov~\cite{KLS15}]\label{lem: bippath}
Let $p=\frac{c}{k}$ for $c$ sufficiently large, and let $G$ be a graph of minimum degree at least $k$. 
If $G$ is bipartite, then $G_p$ a.a.s.\ contains a path of length $\left(2-6c^{-1/2}\right)k$.
\end{lemma}

The next lemma is of a similar flavor as the last one.
We suitably modify a result of~\cite{KLS15} for our purposes. 

\begin{lemma} \label{lem: logk-starting-path}
Let $p=\frac{c}{k}$ for $c$ sufficiently large, and let $G$ be a graph of minimum degree at least $k$. If $V_0\subseteq V(G)$ with $|V_0|\geq \log k$, then $G_p$ $a.a.s.$ contains a path of length $\left(1-2c^{-1/2}\right)k$ which starts at a vertex in $V_0$.
\end{lemma}

\begin{proof}
Let $\epsilon=c^{-1/2}$. 
Let $V_0\subseteq V(G)$, and we may assume that $|V_0|=\left\lceil\log k\right\rceil$. 
We modify the DFS-algorithm as follows. 

Recall that the stack $S$ denotes the vertices that are currently under exploration. 
If $S=\es$ in some step of the algorithm, 
then as long as possible we take a vertex of $V_0\cap U$ as the new root of a component and put it onto the stack $S$. 
Hence, by this modified DFS-algorithm, 
at least up to the point when we explored at most $\log k$ vertices,
the root of the current component is in $V_0$.

We run this modified DFS-algorithm until the moment at which we reach $|R\cup S|=(1-\epsilon)k$. 
Let $\cA$ be the event that $S=\emptyset$ at some moment after $\frac{1}{2}\log k$ steps of the algorithm and 
let $\cB$ be the event that there are less than $(1-\epsilon)k$ positive answers among the first $\frac{k}{p}=\epsilon^2k^2$ tested edges. 

\begin{claim}\label{cla: all fine}
 $\bP[\cA\cup \cB]=o(1).$
\end{claim}

Assuming this claim we can a.a.s.\ find a path of length $(1-\epsilon)k$ starting in a vertex of $V_0$ as follows.

Suppose neither $\cA$ nor $\cB$ holds.
Consider the step of the DFS-algorithm at which we reach $|R\cup S|=(1-\epsilon)k$. 
Thus the root of the current component is contained in $V_0$, as $\cA$ does not hold.
Due to property (\ref{DFS1}) such a step exists. 
Recall that the vertices in $S$ form a path (property (\ref{DFS2})). 
If $|S|\geq(1-2\epsilon)k$, then the statement of the lemma follows directly. 
Thus, we may assume that 
\begin{align}\label{S small}
	|S|<(1-2\epsilon)k
\end{align}
which implies	$|R|>\epsilon k$.
Moreover, each vertex in $R$ has at least 
$k-|R\cup S|\geq \epsilon k$ neighbors in $G$ in the set of unreached vertices $U$. 
Due to property (\ref{DFS3}), all these edges between $R$ and $U$ have been queried and answered negatively. 
Hence at least $|R|\cdot \epsilon k>\epsilon^2 k^2$ queries are answered negatively
and less than $(1-\epsilon)k$ are answered positively.
Thus $\cB$ holds, which is a contradiction.

We complete the proof of the lemma by the proof of claim.
For a positive integer $i$, 
let $\cA_i$ be the event that we complete exploring a component when $|R|=i$.
Since every vertex has degree at least $k$,
in this moment of the algorithm every vertex in $R$ has at least $k-i\geq \epsilon k$ neighbors in $U$ (for $i\leq (1- \epsilon) k$) and all these edges are queried negatively. 
Thus we queried at least $i \epsilon k$ edges in total, and had at most $i$ positive answers.
The probability that this occurs is at most the probability that 
a binomial distributed random variable $X_i$ with $X_i\sim\bin( i\epsilon k, p)$
is at most $i$.
Hence $\bE X_i = i\epsilon c=ic^{1/2}$.
By Chernoff's inequality, we obtain
\begin{align*}
	\bP[\cA_i]
	\leq\bP[X_i\leq i]
	\leq	\bP\left[\Big|X_i- ic^{1/2}\Big|\geq \frac{ic^{1/2}}{2}\right]
	\leq
	2e^{-\frac{ic^{1/2}}{12}}
	\leq \frac{1}{2^i}.
\end{align*}
Using the union bound leads to the desired result
\begin{align*}
	\bP[\cA]
	\leq\bP\left[\bigcup\limits_{i=\frac{1}{2}\ln k}^{(1-\epsilon)k}\cA_i\right]
	\leq
	\sum_{i=\frac{1}{2}\ln k}^{(1-\epsilon)k}
	\bP[\cA_i]
	\leq
	\sum_{i=\frac{1}{2}\ln k}^{(1-\epsilon)k}\frac{1}{2^i}
	=o(1).
\end{align*}
An upper bound for the event $\cB$ follows by a direct applications of Chernoff's inequality.
Let $Y$ be a binomial distributed random variable with $Y\sim\bin\left(\frac{k}{p},p\right)$. 
Then,
\begin{align*}
	\bP[\cB]
	\leq\bP[Y\leq (1-\epsilon)k]
	\leq 
	2\exp\left(-\frac{\epsilon^2 k}{3}\right)=o(1).
\end{align*}
This implies $\bP[\cA\cup \cB]=o(1)$, 
which completes the proof of the claim and thus the proof of the lemma.
\end{proof}

\section{Long Cycles}
In this section we prove Theorem~\ref{thm: cycle}. 
Let $G$ be a graph of minimum degree at least $k$ on $n$ vertices and let $p=\frac{c}{k}$ for $c$ sufficiently large. 

This proof is based on ideas of Riordan~\cite{Rio14} and follows its strategy.
In particular, the first two short lemmas naturally transfer to our setting.

In this section, we consider a rooted forest $T$ which is an output of the DFS-algorithm described in the beginning.
We emphasize that every untested edge of $G$ is in $G_p$ independently of $T$.

\begin{lemma}\label{lem: tested edges}
During the DFS-algorithm on $G_p$ a.a.s.\ at most $\frac{2n}{p}=\frac{2nk}{c}$ many edges are tested.
\end{lemma}
\begin{proof} 
We run the DFS-algorithm on $G_p$. 
Note that the rooted spanning forest $T$ of $G_p$ has at most $n-1$ edges and that every positively answered query contributes an edge to our exploration of this forest. 
Let $X$ be the number of tested edges. 
If at least $\frac{2n}{p}$ many edges are tested, 
then let $Y$ be the number of positively answered queries of the first $\frac{2n}{p}$ tested edges. 
Thus, $Y$ is a binomial distributed random variable with $Y\sim \bin\left(\frac{2n}{p},p\right)$ 
and $\bE Y=\frac{2n}{p}\cdot p=2n$. 
By Chernoff's inequality, we obtain
\begin{align*}
	\bP\left[X>\frac{2n}{p}\right] \leq
	\bP\left[Y<n\right]\leq \bP\big[\vert Y-2n\vert\geq n\big]\leq 2e^{-\frac{n}{6}}=o(1).
\end{align*}
This completes the proof.
\end{proof}

\bigskip
\noindent
From now on, let $\epsilon=c^{-1/5}$. 
Let $E_u$ be the set of untested edges of $G$ during the DFS-algorithm. 
We call a vertex \emph{free} if it is incident with at least $(1-\epsilon)k$ untested edges in $E_u$.

\begin{lemma}\label{lem: all free}
At most $4\epsilon^4n$ vertices of the rooted forest $T$ are a.a.a.\ not free.
\end{lemma}
\begin{proof}
Let $v\in V(T)$ be a vertex that is not free. 
Since the minimum degree of $G$ is at least $k$, the vertex $v$ is incident with at least $\epsilon k$ tested edges. 
Assume that there are more than $4\epsilon^4n$ vertices that are not free. 
Hence, we have more than $\frac{1}{2}4\epsilon^4n \cdot\epsilon k=\frac{2nk}{c}$ many tested edges in total. 
By Lemma~\ref{lem: tested edges}, the probability of this is $o(1)$, which implies the statement.
\end{proof}

For a rooted forest $T$ and a vertex $v\in V(T)$, we introduce the following notation.
\begin{enumerate}[(i)]
	\item Let $A(v)$ be the set of ancestors of $v$ in $T$ excluding $v$ and let $D(v)$ be the set of descendants of $v$ in $T$ excluding $v$.
	\item Let $A_i(v)$ and $D_i(v)$ be the sets of ancestors and descendants of $v$ at distance exactly $i$, respectively, and let $A_{\leq i}(v)$ and $D_{\leq i}(v)$ be the sets of ancestors and descendants of $v$ at distance at most $i$.
	\item The height of the vertex $v$ is defined as $\max\{i:D_i(v)\neq\emptyset\}$.
	\item For two vertices $u,v$,
	let $d(u,v)$ be the number of edges on a shortest $u,v$-path in $T$.
	\item We say a vertex $v$ is \emph{up} if it has many descendants, say if $|D(v)|\geq \epsilon k$.
	If this is not the case, then $v$ is \emph{down}.
	\item\label{heavy} We call the vertex $v$ \emph{skinny} if $|D_{\leq(1-5\epsilon)k}(v)|\leq (1-4\epsilon)k$. 
	Let $Y$ denote the set of vertices in $T$ that are not skinny.
\end{enumerate}

\begin{lemma}\label{lem: heighthk}
If the rooted forest $T$ of $G_p$ contains at most $5\epsilon^4n$ down vertices, 
then, for any constant $h\geq 1$, 
at most $6h\epsilon^3n$  vertices of $T$ are at height less than $hk$.
\end{lemma}
\begin{proof}
For each up vertex $v\in V(T)$, 
let $P(v)$ be a set of $\epsilon k$ descendants of $v$, 
obtained by choosing vertices of $D(v)$ one-by-one starting with those with largest distance to $v$ in $T$. 
For every $w\in P(v)$, we have $|D(w)|<|P(v)|=\epsilon k$, 
because $D(w)\subsetneq P(v)$. 
This implies that every vertex $w\in P(v)$ is down.
 
We define the set $\cS_1=\{ (v,w) : v\text{ is up and }w\in P(v)\}$. 
Each up vertex $v$ appears in exactly $\epsilon k$ pairs $(v,w)\in \mathcal{S}_1$ and by the assumption of the lemma, 
we have at least $(1-5\epsilon^4)n$ up vertices. 
Hence, we obtain
\begin{align*}
|\mathcal{S}_1|\geq \left(1-5\epsilon^4\right)\epsilon kn.
\end{align*}
We consider the pairs $(v,w)\in \mathcal{S}_1$ that satisfy $d(v,w)\leq hk$. 
For pairs $(v,w)\in \mathcal{S}_1$, 
we conclude that $v\in A(w)$ and $w$ is down. 
Note that each vertex has at most one ancestor at each distance, hence $|A_{\leq hk}(w)| \leq hk$. 
Since we have at most $5\epsilon^4n$ down vertices, 
this implies that there are at most $hk\cdot5\epsilon^4n$ pairs $(v,w)\in \mathcal{S}_1$ satisfying $d(v,w)\leq hk$. 
Hence, if we consider the set \mbox{$\mathcal{S}_1^{'}=\left\{ (v,w) \in \mathcal{S}_1: d(v,w)>hk \right\}$}, then
\begin{align*}
	|\mathcal{S}_1^{'}|&\geq |\mathcal{S}_1| -5h\epsilon^4kn\\
	&\geq \left(1-5\epsilon^4\right)\epsilon kn-5h\epsilon^4kn\\
	&\geq \left(1-6h\epsilon^3\right)\epsilon kn.
\end{align*}
Recall that each up vertex $v$ appears in exactly $\epsilon k$ pairs $(v,w)\in \mathcal{S}_1$, and since $\mathcal{S}_1^{'}\subset\mathcal{S}_1$, each such $v$ appears also in at most $\epsilon k$ pairs $(v,w)\in\mathcal{S}_1^{'}$. 
Hence, at least 
\begin{align*}
	\frac{\left(1-6h\epsilon^3\right)\epsilon kn}{\epsilon k}=
	\left(1-6h\epsilon^3\right)n
\end{align*}
distinct up vertices $v$ appear in pairs $(v,w)\in\mathcal{S}_1^{'}$. 
By the definition of $\mathcal{S}_1'$, each such vertex $v$ is at height at least $hk$, which completes the proof. 
\end{proof}

\begin{lemma}\label{lem: longpath}
If the rooted forest $T$ of $G_p$ contains at most $5\epsilon^4n$ down vertices 
and $X\subseteq V(T)$ such that $\vert X\vert\leq 5\epsilon^4n$,
then, for $c$ sufficiently large, $T$ contains a vertical path $P$ of length at least $4k$ containing at most $\frac{1}{4}\epsilon k$ vertices in $X\cup Y$.
\end{lemma}

\begin{proof}
Let $X$ be a subset of $V(T)$ of size at most $5\epsilon^4n$.
First we show that the set $Y\subseteq V(T)$ which contains the vertices that are not skinny is small enough for our purposes. 
We define the set 
\begin{align*}
	\mathcal{S}_2=\{(v,w) : v\in A(w),~ 0<d(v,w)\leq (1-5\epsilon)k \}.
\end{align*}
Since a vertex has at most one ancestor at any given distance, we conclude 
\begin{align*} 
	|\mathcal{S}_2|\leq (1-5\epsilon)kn. 
\end{align*}
By Lemma~\ref{lem: heighthk}, all but at most $6\epsilon^3n$ vertices $v$ are at height at least $k$ and 
thus, each such $v$ appears in at least $(1-5\epsilon)k$ pairs $(v,w)\in\mathcal{S}_2$. 
This contributes at least 
\begin{align*}
	(1-5\epsilon)(1-6\epsilon^3)kn
\end{align*}
pairs to the set $\mathcal{S}_2$. 
Since $|\mathcal{S}_2|\leq (1-5\epsilon)kn$, 
the number of vertices $v$ that appear in more than $(1-4\epsilon)k$ pairs $(v,w)\in \cS_2$ is at most $\left(1-5\epsilon\right) 6\epsilon^2n$,
as (if a vertex $v$ has appears in at least $(1-4\epsilon)k$ pairs $(v,w)$, then it contributes $\epsilon k$ more pairs to the lower bound given before)
\begin{align*}
	\left(1-5\epsilon\right)\left(1-6\epsilon^3\right)kn+
	\left(1-5\epsilon\right) 6\epsilon^2n\cdot \epsilon k=\left(1-5\epsilon\right)kn,
\end{align*}
is an upper bound for $|\mathcal{S}_2|$.

By the definition of $\mathcal{S}_2$
all vertices $v$ appearing in at most $(1-4\epsilon)k$ pairs $(v,w)\in \cS_2$ are skinny.
Hence,
\begin{align*}
	|Y|\leq \left(1-5\epsilon\right) 6\epsilon^2n\leq 6\epsilon^2n.
\end{align*}
Next we want to find the desired path $P$. 
We define the set
\begin{align*}
	\mathcal{S}_3=\left\{(v,w)\colon w\in X\cup Y,~ v\in A(w),~d(v,w)\leq 4 k\right\}.
\end{align*}
Since a vertex has at most one ancestor at each distance, 
for a pair $(v,w)\in\mathcal{S}_3$, 
the vertex $w$ can appear in at most $4 k$ different pairs in $\mathcal{S}_3$. 
We obtain 
\begin{align*}
	|\mathcal{S}_3|&\leq 4k \cdot|X\cup Y|\\
	&\leq 4k\cdot \left(5\epsilon^4n+ 6\epsilon^2n \right)\\
	&\leq 25\epsilon^2kn.
\end{align*}
This implies that the number of vertices $v$ that can appear in more than $\frac{1}{4}\epsilon k$ pairs $(v,w)\in\mathcal{S}_3$, 
is bounded from above by
\begin{align*}
	\frac{25\epsilon^2kn}{\frac{1}{4}\epsilon k}=100\epsilon n.
\end{align*}
By Lemma~\ref{lem: heighthk}, all but at most $24\epsilon^3n$ vertices of $T$ are at height at least $4k$ and from above follows that all but at most $100\epsilon n$ vertices $v$ appear in at most $\frac{1}{4}\epsilon k$ pairs $(v,w)\in\mathcal{S}_3$. 
Hence, for $c$ sufficiently large such that $\epsilon$ is small enough, 
there exists a vertex $v$ at height at least $4k$ that appears in at most $\frac{1}{4}\epsilon k$ pairs $(v,w)\in\mathcal{S}_3$. 
Let $P$ be the vertical path from $v$ to some vertex in $D_{4 k}(v)$. 
Then $P$ has length $4 k$ and by the choice of $v$, 
the path $P$ contains at most $\frac{1}{4}\epsilon k$ vertices in $X\cup Y$.
\end{proof}

\begin{proof}[Proof of Theorem~\ref{thm: cycle}]
Recall, $G$ is a graph of minimum degree at least $k$ and $p=\frac{c}{k}$ for $c$ sufficiently large.

We run the DFS-algorithm on $G_p$. 
Let $T$ be the spanning forest and let $E_u$ be the set of untested edges of $G$ that we obtain from this algorithm. 
By Lemma~\ref{lem: all free}, 
we may assume that all but at most $4\epsilon^4n$ vertices of $T$ are free, 
that is, incident with at least $(1-\epsilon)k$ untested edges in $E_u$. 
Due to property~(\ref{DFS4}) of the DFS-algorithm, 
for every untested edge $uv\in E_u$, either $u\in A(v)$ or $u\in D(v)$.

Assume that for more than $2\log k$ vertices $v$, we have
\begin{align}\label{case1}
	\Big|\left\{u:uv\in E_u,~ d(u,v)\geq(1-5\epsilon)k\right\}\Big| \geq \epsilon k.
\end{align}
This means, that we can find at least $\epsilon k\log k$ untested edges $uv\in E_u$ in $G$ with $d(u,v)\geq(1-5\epsilon)k$. 
Using Chernoff's inequality, 
we can easily find one of these edges present in $G_p$ with probability $1-o(1)$.
As we expect $\epsilon c \log k=\epsilon^{-4} \log k$ edges,
the probability for the event that at least one edge is present
is at least $1-2\exp(-\frac{\epsilon^{-4}\log k}{3})=1-o(1)$.
Thus we can a.a.s.\ find such an edge present in $G_p$ that forms together with $T$ a cycle of length at least $(1-5\epsilon)k$ in $G_p$.

Now assume that for all vertices $v$ except for at most $2\log k$, we have
\begin{align}\label{case2}
	\Big|\left\{u:uv\in E_u,~d(u,v)\geq(1-5\epsilon)k\right\}\Big| < \epsilon k.
\end{align}
Let $V_0$ be the set of vertices $v$ that do not satisfy (\ref{case2}), 
that is, $|V_0|\leq 2\log k$. 

\begin{claim}
A.a.s.\ there are at most $5\epsilon^4n$ down vertices.
\end{claim}
\begin{proof}
Assume that some vertex $v\in V(T)\setminus V_0$ is free and down. 
Since $|D(v)|<\epsilon k$ and $v$ is free, 
there are at least $(1-\epsilon)k-\epsilon k= (1-2\epsilon)k$ pairs of untested edges $uv\in E_u$ with $u\in A(v)$. 
Since each vertex has at most one ancestor at each distance, $v$ has at least $(1-2\epsilon)k-(1-5\epsilon)k=3\epsilon k$ ancestors $u$ with $uv\in E_u$ and $d(u,v)\geq(1-5\epsilon)k$, 
which is a contradiction as $v\notin V_0$. 
Therefore, no down vertex in $V(T)\setminus V_0$ is free. 
By Lemma~\ref{lem: all free}, 
a.a.s.\ all but $4\epsilon^4n$ vertices are free. 
Hence, at most 
\begin{align*}
	4\epsilon^4n+|V_0|\leq 4\epsilon^4n+2\log k\leq 5\epsilon^4n
\end{align*}
vertices are down.
\end{proof}

Thus we may apply Lemma~\ref{lem: longpath}, 
where $X$ is the union of $V_0$ and the set of vertices that are not free,
that is, 
$|X|\leq 5\epsilon^4n$,
and recall that $Y$ is the set of vertices that are not skinny. 
Let $P$ be the path that is given by the Lemma~\ref{lem: longpath} 
and let $Z$ be the set of vertices of $V(P)\sm V_0$ that are free and skinny. 
By Lemma~\ref{lem: longpath}, we obtain
\begin{align*}
	\big|V(P)\setminus Z\big| = \big|(X\cup Y)\cap V(P)\big| \leq \frac{1}{4}\epsilon k.
\end{align*}
For any vertex $v\in Z$, 
there are at least $(1-\epsilon)k$ untested edges $uv\in E_u$ with $u\in A(v)\cup D(v)$. 
We want to show that there are sufficiently many of these vertices $u$ in $A(v)$.

Because of (\ref{case2}) and because $v\in Z$ implies $v\notin V_0$, 
at least $(1-2\epsilon)k$ of these vertices $u$ with $uv\in E_u$ satisfy $d(u,v)\leq(1-5\epsilon)k$. 
Moreover, as $v$ is skinny, at least $(1-2\epsilon)k-(1-4\epsilon)k=2\epsilon k$ vertices $u$ must be ancestors of $v$ 
with $d(u,v)\leq(1-5\epsilon)k$. 
We define a set of ancestors of $v$ within a certain distance, namely 
\begin{align*}
	B(v)=\{u\in A(v): uv\in E_u,~ \epsilon k\leq d(u,v)\leq(1-5\epsilon)k\}.
\end{align*}
Again, since $G$ has only one ancestor at each distance, we obtain $|B(v)|\geq \epsilon k$.  

Let $u_1\in V(P)$ be the vertex on the path $P$, 
which is at height $k$. 
Let $V_1$ be the set of the first descendants of $u_1$ on $P$, 
such that $\big|V_1\cap Z\big|\geq \log k$.  
Since $|V(P)\setminus Z| \leq \frac{1}{4}\epsilon k$, 
we have
\begin{align*}
	V_1\subset D_{\leq\frac{1}{4}\epsilon k+\log k}(v) \cap V(P).
\end{align*}
For each of these vertices $v\in V_1\cap Z$, 
we have $|B(v)|\geq \epsilon k$. 
Hence, there are at least $\epsilon k\log k$ untested edges $uv\in E_u$ such that $v\in V_1\cap Z$ and $u\in B(v)$. 
Using Chernoff's inequality similar as before, 
there is an edge $v_1u_2$ present in $G_p$ such that 
$v_1 \in V_1$, $u_2\in B(v_1)$ and $\epsilon k\leq d(v_1,u_2)\leq (1-5\epsilon)k$ with probability $1-o(k^{-1})$.

Let $V_2$ be the set of the first descendants of $u_2$ on $P$ 
such that $|V_2\cap Z|\geq \log k$. 
Thus for every vertex $w\in V_2$,
we have $d(w,u_1)\geq  \epsilon k -\frac{1}{4}\epsilon k -  2\log k > \frac{\epsilon}{2}k$.
 
Again, as $\big|V_2\cap Z\big|\geq \log k$, 
and there is an edge $v_2u_3$ present in $G_p$ with $v_2\in V_2$ and $u_3\in B(v_2)$ with probability $1-o(k^{-1})$.
 
Next, let $V_3$ be the set of the first descendants of $u_3$ on $P$ such that $\big|V_3\cap Z\big|\geq \log k$.

We may continue in this manner to find such edges $v_iu_{i+1}$ until we reach a vertex $u_{j+1}$ which is at least $2k$ steps higher than the vertex $v_1$.
Since each vertex $v_{i+1}$ is at least $\frac{1}{2}\epsilon k$ steps above $v_i$, 
after at most $4\epsilon^{-1}$ many steps we reach the vertex $u_{j+1}$, 
that is, $j\leq 4\epsilon^{-1}$. 
Thus the procedure does not fail with  probability $1-o(\epsilon^{-1} k^{-1})=1-o(1)$.

Note that we also remain within the path $P$, since $P$ has length at least $4k$ and we start at most at height $k$ and with each step we go up at most $(1-5\epsilon)k$.

Suppose $j$ is even.
Consider the following cycle $C$: 
$$v_1u_2Pv_3u_4Pv_5u_6Pv_7 \ldots v_ju_{j+1}Pu_jv_{j-1}Pu_{j-2} \ldots u_2v_1.$$
Note that every vertex in $V(P)\sm V(v_1Pu_{j+1})$ is contained in some $V_i$.
Therefore, the length of $C$ is at least
\begin{align*}
	2k - 4\epsilon^{-1}\log k - \frac{1}{4}\epsilon k>k.
\end{align*}
A similar argument applies if $j$ is odd.
\end{proof}

\section{Long Cycles in Pseudo-Cliques}\label{sec:pseudo-clique}

Consider the well-known $G(n,p)$-model and with our notation a uniform at random chosen member is $(K_n)_p$.
It is very natural and intuitive that 
$(K_n)_p$ and $H_p$ typically have the same properties
if $H$ is a graph on $n$ vertices which is almost a clique.
In this section we indicate that a result of Frieze~\cite{Fri86} can be suitably modified.

Let $\gamma>0$ be a constant sufficiently small.
We call a graph $G$ on $n$ vertices a \emph{$k$-pseudo-clique} (or simply pseudo-clique) if its minimum degree is at least $k$ and $n\leq (1+ \gamma)k$.
We start with some properties of a pseudo-clique $G$, but before we need to introduce some notation.

A vertex $v$ has \emph{small} degree if $d(G)\leq \frac{c}{10}$ and otherwise its degree is \emph{large}.
Let $S$ and $L$ be the set of all vertices of small and large degree in $G_p$, respectively.
For $1\leq i\leq 4$, 
let $W_i$ be the set of all vertices $v$ of small degree
such that there is a vertex $w$ of small degree and a $v,w$-path of length $i$
or $v$ is contained in a cycle of length $i$.
We set $W=W_1 \cup \ldots \cup W_4$.

The following lemmas are extensions of the results of Frieze~\cite{Fri86},
who prove the analogous results for $G=K_{k+1}$.
As the proofs are quite standard, a bit tedious and can be done along the lines of the proofs of Frieze,
we omit the proofs.

\begin{lemma}
Let $G$ be a $k$-pseudo-clique on $n$ vertices, $p=\frac{c}{k}$ and let $\ell\geq 7$ be an integer.
Then a.a.s.\ $G_p$ has the following properties,
\begin{enumerate}[(a)]
	\item $|\{v\in V(G): d_{G_p}(v)\leq \frac{c}{10}+1\}|\leq (1+ \gamma)ke^{-\frac{2}{3}c}$,
	\item for all sets $Z\subset V(G)$ with $|Z|\geq ke^{-c}$, 
	we have $|\{e\in E(G_p): e\cap S\not= \es\}|\leq 4c|S|$,
	\item $\Delta(G_p)\leq 4\log k$,
	\item $|W|\leq c^4 e^{- \frac{4c}{3}}k$,
	\item $\es \not= Z \subseteq L$ and $|Z|\leq \frac{k}{2\ell}$ implies $|N_{G_p}(Z)|\geq \ell |Z|$, and
	\item $ Z\subseteq V(G)$ and $ \frac{k}{2\ell}\leq |Z|\leq \frac{1}{2}k$ implies $m(G_p[Z])\geq \frac{c|Z|}{3\ell}$.
\end{enumerate}
\end{lemma}

\begin{lemma}
Let $G$ be a $k$-pseudo-clique on $n$ vertices, $p=\frac{c}{k}$, and
let $X_1,X_2,\ldots $ be a sequence obtained by the following rule
\begin{align*}
	X_i = \left\{ v\in V(G): \left|N_{G_p}(v) \cap \left(S\cup \bigcup_{j=1}^{i-1}X_j\right)\right|\geq 2\right\}.
\end{align*}
If $X=\bigcup_{j\geq 1}X_j$,
then $|X|\leq 500c^4e^{- \frac{4c}{3}}k$ a.a.s. 
\end{lemma}

Let $V_2$ be vertex set of the largest subgraph of $G_p$ with minimum degree $2$ ($G_p[V_2]$ is also known as the $2$-core).
Moreover, let $Y$ be the set of all vertices $v$ in $G$ which have degree $2$ and have a neighbor in $X$ in $G_p$.
Let $A=V_2\sm (W \cup X \cup Y)$.
\begin{lemma}
Let $G$ be a $k$-pseudo-clique on $n$ vertices and $p=\frac{c}{k}$.
Then, a.a.s.
\begin{align*}
	|A|\geq \left(1- (1+\epsilon(c))ce^{-c}\right)k,
\end{align*}
where $\epsilon(c) \rightarrow 0$ as $c \rightarrow\infty$.
\end{lemma}

Having proved these three lemmas for pseudo-cliques,
one can go once again along the lines of the result of Frieze to obtain the following.

\begin{theorem}\label{thm: pseudo-clique}
If $G$ be a $k$-pseudo-clique on $n$ vertices and $p=\frac{c}{k}$,
then a.a.s.\ $G_p$ contains a cycle of length at least
\begin{align*}
	\left(1- (1+\epsilon(c))ce^{-c}\right)k,
\end{align*}
where $\epsilon(c)\rightarrow 0$ as $c \rightarrow\infty$.
\end{theorem}

\section{Long Paths}

This section is devoted to the proof of Theorem~\ref{thm: path}.
This proof is inspired by a result in \cite{KLS15} proving that a.a.s.\ the random subgraph $G_p$ of a graph $G$ of minimum degree at least $k$ contains a path of length $k$ if $p=\frac{(1+ \epsilon)\log k}{k}$ for any fixed $\epsilon>0$.

\begin{proof}[Proof of Theorem~\ref{thm: path}]
Let $c$ be sufficiently large and let $\epsilon=5\left(\frac{c}{3}\right)^{-1/5}$.
If $G$ contains a set $V'\subseteq V(G)$ such that 
\begin{align}\label{eq: dense set}
	\left(1- \frac{1}{\log k}\right)k \leq |V'| \leq (1+ 10\epsilon)k
\end{align}
and the minimum degree of the graph $G[V']$ is at least $(1- \frac{2}{\log k})k$ , then by Theorem~\ref{thm: pseudo-clique}, 
$G_p$ a.a.s.\ contains a cycle of length at least 
\begin{align*}
	\left(1- (1+ \epsilon(c))ce^{-c}\right)k,
\end{align*}
for some function $\epsilon(c) \rightarrow 0$ as $c \rightarrow \infty$, which implies the statement. 
Hence, we may assume that $G$ does not contain such a set $V'$.

In the following, we use a technique which is known as \emph{sprinkling}.
In our case, we expose the edges of $G_p$ in three rounds
and in each round we suppose an edge to be present independently with probability $\frac{c}{3k}$.
Thus we consider the union of three graphs $G_{p_1}\cup G_{p_2}\cup G_{p_3}$, where $p_i=\frac{c}{3k}$.
As
\begin{align*}
	1-(1-p_1)(1-p_2)(1-p_3)=\left(1-\frac{c}{3k}+\frac{c^2}{27k^2}\right)\frac{c}{k}\leq p,
\end{align*}
the union of these three graphs underestimates the model $G_p$.
Therefore, if we can show that $G_{p_1}\cup G_{p_2}\cup G_{p_3}$ a.a.s.\ contains a path of the desired length, 
then also $G_p$ a.a.s.\ contains such a path. 

By Theorem \ref{thm: cycle}, 
we know that $G_{p_1}$ a.a.s.\ contains a cycle $C$ of length at least $(1-\epsilon)k$. 
Moreover, we may assume that $|C|<\left(1- (1+\epsilon(c))ce^{-c}\right)k$. 
Let $A\subseteq V(G)\setminus V(C)$ be the set of vertices having at least $(1-20\epsilon)k$ neighbors in $V(C)$
and let $B= V(G)\sm (V(C)\cup A)$.

We divide the proof into two parts.
First, we suppose that $|A|\leq 10\epsilon k$.
Hence, if $B\not=\es$, then $G[B]$ has minimum degree at least $10\epsilon k$.

Suppose first that at least $4k\log k$ edges joining $C$ and $B$ in $G$ and denote this set by $E$.
Consider an ordering $b_1,b_2,\ldots$ of the vertices in $B$
and consider an ordering $e_1,e_2,\ldots$ of the edges in $E$ which respects the ordering on $C$,
that is, if $i<j$, then the indices of the edges incident to $b_i$ are smaller than the indices of the edges incident to $b_j$.
For $1\leq i\leq \lceil2\log k\rceil$, let $E_i=\{e_j: (2i-2)k+1 \leq j \leq (2i-1)k\}$.
This implies that there is no vertex $b\in B$ incident to an edge in $E_i$ and $E_j$ for $i\not=j$, 
since a vertex in $B$ has at most $|V(C)|\leq k$ neighbors in $C$.
Moreover, with probability $1- e^{- \frac{c}{3}}$ every set $E_i$ contains at least one edge in $G_{p_2}$ independently for every $i$.
Thus by Chernoff's inequality, at least $\log k$ sets $E_i$ contain an edge in $G_{p_2}$ with probability $1-o(1)$.
Let $S$ be a set of $\log k$ vertices in $B$ incident to an edge in $G_{p_2}$.
By Lemma~\ref{lem: logk-starting-path}, with probability $1-o(1)$, there is a path in $G_{p_3}[B]$ starting in $S$ of length, say, $\epsilon k$.
Combining $C$, a suitable edge in some $E_i$, and this path leads to a path in $G_{p_1}\cup G_{p_2}\cup G_{p_3}$ of length at least $k$ with probability $1-o(1)$.

Therefore, we may assume that at most $4k\log k$ edges joining $C$ and $B$ in $G$.
Hence 
\begin{align*}
	|A \cup C|\geq k - 5\log k,
\end{align*}
otherwise every vertex in $C$ has at least $5\log k$ neighbors in $B$ contradicting our assumption.

Next, we suppose  that there exists a set $A'\subseteq A$ with at least $\sqrt{k}$ many vertices having at least $k^{\frac{2}{3}}$ many neighbors in $B$.
As any vertex in $A'$ is adjacent to at least one vertex in $C$ in $G_{p_2}$ with probability close to $1$, say $\frac{3}{4}$, independently of each other,
with probability $1-o(1)$, 
there exists a $A''$ of size at least $\frac{|A'|}{2}$ 
such that every vertex in $A''$ is adjacent to $C$ in $G_{p_2}$. 
By a similar argument as before, 
with probability $1-o(1)$, 
there are $\log k$ vertices in $B$
such that each of them has a neighbor in $A''$ in $G_{p_2}$.
Again, with probability $1-o(1)$,
there is a path in $G_{p_3}$ at length at least $\epsilon k$ starting in one of these vertices in $B$ and this leads to a path of length at least $k$ in $G_{p_1}\cup G_{p_2}\cup G_{p_3}$ with probability $1-o(1)$.

Therefore,
there are at most $2\sqrt{k}$ vertices $v$ in $A \cup C$ with $d_{B}(v)\geq k^{\frac{2}{3}}$
and let $Z$ be obtained from $A \cup C$ by deleting all these vertices.
Clearly, $|Z|\geq k- 2\sqrt{k}$.
As $|Z|\leq (1+ 10\epsilon)k$, the set $Z$ is a set as in \eqref{eq: dense set}, which is a contradiction.

Thus from now on, we may assume that $|A|\geq 10 \epsilon k$.
Let $A_1\subseteq A$ with $|A_1|= 10 \epsilon k$.
We partition $C$ into $\frac{1}{10\epsilon}$ cycle segments $S_1,S_2,\ldots$ each of length almost $10 \epsilon k$.
As every vertex in $A_1$ has at least $(1- 20\epsilon)k$ neighbors in $C$,
by a simple average argument,
there is a segment, say $S_1$, such that the number of edges between $S_1$ and $A_1$ is at least $(1- 20 \epsilon)|A_1||S_1|$.
Let $H$ be the bipartite subgraph of $G$ which is induced by $A_1$ and $S_1$.
This implies that the bipartite complement of $H$ has at most $2000 \epsilon^3 k^2$ edges.
Of course, this graph contains at most $100\epsilon^{\frac{3}{2}} k$ vertices of degree at least $100\epsilon^{\frac{3}{2}}k$.
Let $H'$ be the graph obtained by deleting these vertices from $H$.
Thus $H'$ has minimum degree at least $(1-20\sqrt{\epsilon})\cdot 10 \epsilon k$.

For some orientation of $C$, 
let $L$ and $R$ be the first and last $\epsilon k$ vertices on $C$ in $S_1$.
Moreover, remove an arbitrary subset of $A_1$ to obtain from the graph $H'\sm (R \cup L)$ a balanced bipartite graph $H''$.
Thus $H''$ has minimum degree at least $(1-25\sqrt{\epsilon})\cdot 8 \epsilon k$.

By Lemma~\ref{lem: bippath},
$H''$ contains a path $P$ of length $15\epsilon k$ in $G_{p_2}$ with probability $1-o(1)$.
Let $P_1$ and $P_2$ be the subpaths at the beginning and at the end of $P$ of length $\epsilon k$, respectively.
By Chernoff's inequality, with probability $1-o(1)$, in $G_{p_3}$, there exists an edge $e_1$ joining a vertex in $L$ and $V(P_1)\cap A_1$ 
and an edge $e_2$ joining a vertex in $R$ and $V(P_2)\cap A_1$.
 
Combining the subpath of $C$ between the endpoints of $e_1$ and $e_2$ that contains the segment $S_2$, 
the subpath of $P$ between the endpoints of $e_1$ and $e_2$, 
and the edges $e_1$ and $e_2$
 results in a cycle in $G_{p_1}\cup G_{p_2}\cup G_{p_3}$ of length
at least $(1- 11\epsilon)k+ 13 \epsilon k \geq k$ and this completes the proof.
\end{proof}

\bibliographystyle{amsplain}
\bibliography{randomgraphs}

\vfill

\small
\vskip2mm plus 1fill
\noindent
Version \today{}
\bigbreak

\noindent
Stefan Ehard
{\tt <stefan.ehard@uni-ulm.de>}\\
Universit\"at Ulm, Ulm\\
Germany\\

\noindent
Felix Joos
{\tt <f.joos@bham.ac.uk>}\\
School of Mathematics, University of Birmingham, Birmingham\\
United Kingdom
\end{document}